\documentclass{amsart}
\usepackage{graphicx}
\usepackage{latexsym}
\usepackage{color}
\usepackage{amsfonts}
\usepackage[all]{xy}
\usepackage{amssymb, mathrsfs, amsfonts, amsmath}
\usepackage{amsbsy}
\usepackage{xcolor}
\usepackage{amsmath
	,amsthm
	,amsfonts
	,amssymb
	,amscd
} 
\usepackage[mathscr]{eucal}
\usepackage{multicol}
\usepackage[bottom]{footmisc} 
\usepackage{dsfont}
\usepackage{float} 
	\usepackage{graphicx} 

\usepackage{amsfonts}
  \newtheorem{theo}{Theorem}[section]		
\newtheorem{propo}[theo]{Proposition}

\theoremstyle{definition}
\newtheorem{defi}{Definition}
\theoremstyle{remark}
\newtheorem{remark}[theo]{Remark}    
\newtheorem{exa}[theo]{Example}

\newcommand{\be}{\begin{equation}}
	\newcommand{\ee}{\end{equation}}
\newcommand{\ben}{\begin{enumerate}}
	\newcommand{\een}{\end{enumerate}}
\newcommand{\beq}{\begin{eqnarray}}
	\newcommand{\eeq}{\end{eqnarray}}
\newcommand{\beqn}{\begin{eqnarray*}}
	\newcommand{\eeqn}{\end{eqnarray*}}

\begin{document}
	\title{On Funk's parabolas}
	\author{Newton Sol\'orzano, Junior Moyses and V\'ictor Le\'on}
	
\address{N. Sol\'orzano. ILACVN - CICN, Universidade Federal da Integração Latino-Americana, Parque tecnológico de Itaipu, Foz do Iguaçu-PR, 85867-970 - Brazil}
	\email{nmayer159@gmail.com}

\address{J. Moyses. ILACVN - CICN, Universidade Federal da Integração Latino-Americana, Parque tecnológico de Itaipu, Foz do Iguaçu-PR, 85867-970 - Brazil}
\email{junior.rmoyses@gmail.com}

\address{V. Le\'on. ILACVN - CICN, Universidade Federal da Integração Latino-Americana, Parque tecnológico de Itaipu, Foz do Iguaçu-PR, 85867-970 - Brazil}
\email{victor.leon@unila.edu.br}

	

	\begin{abstract}
We study parabolas in the two dimensional unit disk equipped with a Funk metric. Four types of parabolas are obtained, due to the non-reversibility of the Funk metric, each one with applications to physics in the Zermelo navigation problem. We show that two of the four parabolas obtained are well known conics, and the remaining two are characterized by irreducible quartics. Explicit examples are given.
	\end{abstract}
	
	\date{\today}

	\maketitle
\section{Introduction}

Some preliminary work regarding Finsler metrics was carried out in \cite{Chavez2021} where the authors considered a lake in the form of the unit disks $\mathbb{B}^2,$ with a concentric and symmetric wind current given by the vector field $W(x_1,\,x_2)=(-x_1,\,-x_2)$. The distance function (time of displacement) in this framework, is given by 
\begin{align}\label{111}
	d_F(P,Q)=\ln\left(\frac{\sqrt{\langle P, Q-P \rangle^2+(1-\Vert P\Vert^2)\Vert Q-P\Vert^2}-\langle P, Q-P \rangle}{\sqrt{\langle P, Q-P\rangle^2+(1-\Vert P\Vert^2)\Vert Q-P\Vert^2}-\langle Q, Q-P \rangle}
	\right).   
\end{align}
In \cite{Chavez2021} were also obtained equations for the circle, and formulas for the distance from a point to a line and from a line to a point.\par

Some questions naturally arise in this model: Imagine an isle and a coast in the form of a straight line. Suppose a boat parting from the isle to a point in the beach (or the contrary), but it has to stop in some middle point along the way to fill its tank with fuel.\par
Now imagine two boats wanting to meeting in some point at the ocean, one coming from the beach (of the coast in straight line ) and the other one from the isle, both traveling with the same speed and arriving at the same time, then they should meet in a mid point between them.\par
Finally it can be thought of two boats parting from the same point, one arriving to an isle, the other to the beach at the same time.\par
The physics problems presented above motivate the study of an equivalent to parabolas in the Euclidean plane, that we will call Funk's Parabolas (see Definition \ref{def_parabola}). Note that in each problem we must consider the two directions because, due to the wind, there is no symmetry going from one point to another.\par
Once we fix the external force as been concentric and symmetric, the problems aforementioned involve situations that can be solved with the help of Funk's Parabolas.

\section{Preliminary}

In this section some definitions and results, required for the development of our work, are introduced. We adopted the definitions given in \cite{Chavez2021}, which can also be found in \cite{Manfredo1}, such as inner product, norm, regular curve, arc length (which will be called usual or Euclidean) and vector field.


%
%
%
%


\begin{defi}\label{FunkMetric}{\em 
	Let $ x=(x_1,x_2)\in\Omega=\mathbb{B}^2$ and $y=(y_1,y_2)\in\mathbb{R}^2.$ 
	The function 
	\begin{align*}
		F=&\sqrt{a_{11}(x)y_1^2+ 2a_{12}(x)y_1y_2 + a_{22}(x)y^2_2}+b_1(x)y_1+b_2(x)y_2,
	\end{align*}
	where $b_1= \dfrac{x_1}{1-x_1^2-x_2^2}$, $b_2= \dfrac{x_2}{1-x_1^2-x_2^2}$
	and
	\begin{displaymath}
		[a_{ij}]=\frac{1}{(1-x_1^2-x_2^2)^2}\left(\begin{array}{cc}
			{1-x_2^2}  & {x_1x_2} \\
			{x_1x_2} & {1-x_1^2}
		\end{array}\right),
	\end{displaymath}
	it is called \textit{Funk metric on the unit disks} $ \mathbb{B}^2=\{x\in\mathbb{R}^2;\;\Vert x\Vert<1\}$.}
\end{defi}

This metric models a boat sailing with a unitary speed in $\mathbb{B}^2$, where a wind current given by $W_x=(-x_1,-x_2)$ is present with speed lower that $1$. (see Section 3 in \cite{Chavez2021})

The Funk metric on $\mathbb{B}^2, $ is a special case of Randers metrics. Some properties where studied in \cite{Shen2001} for more general cases. 

\begin{defi}[Funk's type Arc Length]{\em		
	Let $c:[a,b]\rightarrow \Omega\subset\mathbb{R}^2$ a piecewise regular curve. The \textit{arc length (Funk's type) of $c$} is defined by
	\[ \mathscr{L}_F(c):=\int_a^bF(c(t),c'(t))dt.\]		
	For any points $p, q \in \Omega \subset \mathbb{R}^2,$ we define the \textit{distance} from $p$ to $q$ induced by $F$, as
	\[ d_F(p,q):=\operatorname{inf}_c\mathscr{L}_F(c),\]
	where the infimum is taken over the set of all piecewise regular curves $c$ such that $c(a)=p$ and $ c(b)=q $.}	
\end{defi}

It is known that the \textit{shortest paths} in this metric are straight lines (see Example 9.2.1 in \cite{Shen2001}), i.e., the optimal path is always a straight line. This property makes this metric manageable in relation to distance calculation. More details about geodesics and its relation with shortest paths can be found in Section 3.2 in \cite{ChernRF2005}, and  Section 2.3 in \cite{Cheng2012}.\par

In Remark 5.1 in \cite{Chavez2021} was proved that $d_F$ given by \eqref{111} is non-reversible ($ d_F(P,Q)\neq d_F(Q,P) $) and it is non invariant by translations, but it is invariant by rotations.

If we consider \begin{align}\label{eq:def_r}
r=&\frac{\sqrt{k}-\langle P,Q-P\rangle}{\sqrt{k}-\langle Q,Q-P\rangle},
\end{align} where
\[ k = \langle P, Q-P \rangle^2+(1-\Vert P\Vert^2)\Vert Q-P\Vert^2,\]
the equation \eqref{111} can be more manageable using the following version of Theorem 5.1 in \cite{Chavez2021}. 
\begin{theo}\label{maintheorem}{\em 
	Let $P,\;Q$ be points in $\mathbb{B}^2$ and $r\geq1$ a real number, then: 
	\begin{align}\label{13}
		\left\Vert \dfrac{P}{r}-Q\right\Vert&=\dfrac{r-1}{r}
	\end{align}
	it is equivalent to
	$d_F (P, Q) = \ln r.$
}
\end{theo}

It is worth mentioning that the previous theorem can be easily generalized for points $P$ and $Q$ in $\mathbb{R}^n$ at the unit ball (see the proof of the theorem in \cite{Chavez2021}).
In order to facilitate calculations we define the function $ \rho:[0,1)\times [0,1)\to \mathbb{R} $ as,
\begin{align}\label{rxy}
\rho(\epsilon, \eta)=\frac{1+\operatorname{sgn}(\epsilon-\eta)\epsilon}{1+\operatorname{sgn}(\epsilon-\eta)\eta}
\end{align}
where $\operatorname{sgn}$ is the function sign:
\[
\operatorname{sgn}(x) = \left\{
\begin{array}{rll} 
1, & \hbox{if} & x>0\\
0, & \hbox{if} & x=0\\
-1, & \hbox{if} & x<0.
\end{array}
\right.\]
The following formulas can be found in Section 5.2 e 5.3 in \cite{Chavez2021}, and play a fundamental role in the present work:

\begin{propo}[\cite{Chavez2021}]\label{dist_ponto_reta_const}{\em 
Let $P=(a,b)$ be a point and $s:\;x_2=c$ a constant line. The Funk distance $d_F(P,s)$ from the point $P$ to the line $s$ and Funk distance $d_F(s,P)$ from the line $s$ to the point $P,$ are, respectively 
\begin{align}
	d_{F}(P,s)=&d_F\left(P,\left(\frac{a}{\rho(b,c)},c\right)\right)=\ln\left(\rho(b,c)\right),\label{dist_pt_reta_const}\\
	d_{F}(s,P)=&d_F\left(\left(a\cdot \rho(c,b),c\right),P\right)=\ln\left(\rho(c,b)\right),\label{dist_reta_const_pt}
\end{align}
where $\rho(\epsilon,\eta)$ is given by \eqref{rxy}.
}\end{propo}


\section{Funk's Parabolas in $\mathbb{B}^2$}

As the Funk distance is non-reversible, we obtain for types of parabolas.

\begin{defi} \label{def_parabola}{\rm 

Let $s \subset \mathbb{B}^2$ be a fixed line, called \textit{the directrix}, and $\mathcal{F}$ a point in $\mathbb{B}^2 \subset \mathbb{R}^2$, called \textit{the focus}. We define a \textit{Funk's Parabola}, with directrix $s$ and focus $\mathcal{F}$, as the set of points $P \in \mathbb{B}^2$ that satisfies one of the following properties:
\begin{multicols}{2}
\begin{enumerate}
	\item[(P1)] $d_F(\mathcal{F},P)=d_F(s,P);$
	\item[(P2)] $d_F(P,\mathcal{F})=d_F(P,s);$
	\item[(P3)] $d_F(\mathcal{F},P)=d_F(P,s);$
	\item[(P4)] $d_F(s,P)=d_F(P,\mathcal{F}).$
\end{enumerate}
\end{multicols}
}
\end{defi}
The set of point satisfying (P$k$) is called {\bf Funk's parabola of type {\boldmath$k$}} with $k=1,2,3,4$.

\begin{remark} Note that in the previous definition we are also considering the degenerate case, i.e., when the focus $\mathcal{F}$          belongs to the directrix $s$.
\end{remark}
Due to the invariance of the Funk distance by rotations, we will consider first {\bf the case where the directrix is parallel to the  {\boldmath$x$} axis}.\par 
Throughout this work we will use the following notations 
\[ s: y=y_0,\quad \mathcal{F}=(f_0,\,g_0)\;\text{ and }\;P=(x,\,y). \]    
\subsubsection{Funk's Parabola type 1 - Particular Case}  
By Theorem \ref{maintheorem} and \eqref{dist_reta_const_pt} we have that $P=(x,\,y) $ satisfies $P1$ if, and only if, \begin{align}\label{eq:tipe1norm}
\displaystyle\left\Vert \frac{\mathcal{F}}{\rho(y_0,y)}-P\right\Vert^2 = \left\Vert \frac{Q}{\rho(y_0,y)}-P \right\Vert^2,
\end{align}
where $ Q=\left(x.\rho(y_0,y),y_0\right)$ and $\rho$ is given by \eqref{rxy}. 

First, suppose the degenerate case, when the focus $\mathcal{F}=(f_0,g_0)$ is in the line $s=y_0$, i.e., when $y_0=g_0$. Therefore the equation \eqref{eq:tipe1norm} is equivalent to \[x-f_0\cdot (\rho(y_0,y))^{-1}=0.\]
This implies that $P=(x,\,y)$, is either in the line segment $ (1-y_0)x-f_0(1-y)=0, $ for $ y\geq y_0, $ or in the line segment $(1+y_0)x-f_0(1+y)=0, $ for $ y<y_0$. If $f_0=0$, then $ P $ is in the line $x=0$. If $f_0\neq 0$ and $ x $ goes to zero, thus $ y $ approximates to $1$ or $-1$ independently of $f_0$ value.\par 
In what follows we characterize non degenerate Funk's Parabola of type 1.
\begin{propo}
Let   $s: y=y_0 \subset \mathbb{B}^2$ be the directrix, $\mathcal{F}=(f_0,g_0) \in \mathbb{B}^2$ be the focus with $ g_0\neq y_0$ and $ \mathcal{E} \subset \mathbb{R}^2$ an Euclidean ellipse described by
\begin{align}\label{p1}
x^2 +B{x}\overline{y} +C{\overline{y}^2} +E{\overline{y}} = 0,
\end{align}
where 
\begin{align}
\overline{y}=& 1+\operatorname{sgn}(y_0-g_0)y\\
B=&-\frac{2f_0}{\sigma_1},\label{defB}\\
C=&\frac{f_0^2 + (y_0-g_0)^2+2\vert y_0-g_0\vert}{\sigma_1^2},\label{defC}\\
E=&-\frac{2\vert y_0-g_0\vert}{\sigma_1},\label{defE}
\end{align}
and 
\begin{align}\label{defsigma}
\sigma_1= 1+\operatorname{sgn}(y_0-g_0)y_0.
\end{align}
Hence the locus of points of the Funk's Parabola of type 1 is given by
\[ \mathcal{E} - \{(0,-\operatorname{sgn}(y_0-g_0))\}. \] 
\end{propo}

\begin{proof}
The line $s$ divides $\mathbb{B}^2$ in two domains. Suppose $\mathcal{F}\notin s$. We will show that the Funk's Parabola of type 1 and the focus $\mathcal{F}$ are in the same subspace generated by $s$. In fact, suppose that $\mathcal{F}$ and $P$ are in different subspaces generated by $s$.  Let $Q$ and $Q_0$ be in $s$, such that $d_F(s,P)=d_F(Q,P) $ and $ Q_0$ is in the line segment $ \overline{\mathcal{F}P}, $ then
\[d(s,P)=d(Q,P)\leq d(Q_0,P)<d(\mathcal{F},Q_0)+ d(Q_0,P) = d(\mathcal{F},P).\]       
Suppose that there is a point $P=(x,\,y)$ in the Funk's parabola of type 1 in the line $s$, i.e., there is a $y$ such that $y=y_0$, hence by (P1) we have that $ \mathcal{F}=P$.
Therefore, the sign of $ y_0-y $ is equal to the sign of $ y_0-g_0 $, thus we have that $\rho(y_0,y)$ can be written as 
\[\rho(y_0,y)=\frac{1+\operatorname{sgn}(y_0-g_0)y_0}{1+\operatorname{sgn}(y_0-g_0)y}.\]            
Using the properties of absolute value and inner product, equation \eqref{eq:tipe1norm} can be written as
\begin{align}\label{p0}
2\rho(y_0,y)\langle P, \mathcal{F}-Q\rangle + \Vert Q\Vert^{2} - \Vert \mathcal{F}\Vert^{2} = 0.
\end{align}
Due to the fact that $ \operatorname{sgn}(y_0-g_0)\cdot(y_0-g_0)=\vert y_0-g_0\vert $, and replacing $ \mathcal{F}, P$ and $ Q$, we have that equation \eqref{p0} is equivalent to:
\begin{align*}
x^2 -2f_0\frac{x}{\rho(y_0,y)} +\left[f_0^2 + (y_0-g_0)^2+2\vert y_0-g_0\vert\right]\frac{1}{\rho^2(y_0,y)} - 2\vert y_0-g_0\vert\frac{1}{\rho(y_0,y)} = 0,
\end{align*}
which is just the equation \eqref{p1}.

Equation \eqref{p1} is a second-degree equation in $x$ and $1+\operatorname{sgn}(y_0-g_0)y$, thus in $x$ and $y$, and whose discriminant $I$ is given by
\[I=B^2-4C=-\frac{4}{\sigma_1^2}[(y_0-g_0)^2+2\vert y_0-g_0\vert]<0 \]
what show us that \eqref{p1} is elliptical in Euclidean geometry (i.e., the empty set, an ellipse, a point or a circle in $\mathbb{R}^2 $).\par
Note that the point $(0,-\operatorname{sgn}(y_0-g_0))$ satisfies the equation \eqref{p1}. Then, the locus of \eqref{p1} can not be a closed curve. We will see that in fact it is an Euclidean ellipse without the point $(0,-\operatorname{sgn}(y_0-g_0))$.\par
Techniques for determine the conic from a given general equation of second degree can be found in \cite{EdwardsPenney1990,Leithold1976}.\\ 
{\bf Case 1:} {\boldmath$f_0=0$.}
When $f_0=0$, then the change of variables 
\begin{align}
x=&\overline{x},\label{xf0}\\
y=&\left(y-\frac{E}{2C}-1\right)\operatorname{sgn}(y_0-g_0)\label{yf0}
\end{align}
where $\frac{E}{2C} = \frac{\sigma}{2+\vert y_0-g_0\vert}$, transforms the equation \eqref{p1} in the canonical form:
\begin{align}\label{canonicaf0}
\frac{\overline{x}^2}{\frac{E^2}{4C}} + \frac{\overline{y}^2}{\frac{E^2}{4C^2}}=1.
\end{align}
From \eqref{xf0},  \eqref{yf0} and \eqref{canonicaf0}, we have that the Funk's parabola of type 1 is a section of an Euclidean ellipse with center $ \mathcal{C}=\left(0, \frac{g_0-\operatorname{sgn}(y_0-g_0)}{2+\vert y_0-g_0\vert}\right) $ and vertices: 
$ (\pm \frac{E}{2\sqrt{C}},-(1+\frac{E}{2C})\operatorname{sgn}(y_0-g_0)), (0, -(1+\frac{E}{2C}\pm\frac{E}{2C})).$\\
{\bf Case 2:} {\boldmath$f_0\neq 0$}. In \eqref{p1}, consider the change of variable 
\begin{align}
x=&\alpha \overline{x} - \beta \overline{y}-\left(\alpha\frac{D}{2A} + \beta\frac{E}{2C}\right),\label{eq:xf0neq0}\\
1+\operatorname{sgn}(y_0-g_0)y=&\beta \overline{x} + \alpha\overline{y} - \left(\beta\frac{D}{2A}+\alpha\frac{E}{2C}\right),
\end{align} where $ \alpha=\frac{1}{\sqrt{2}}\sqrt{1-\operatorname{sgn}(f)\frac{(1-C)}{\sqrt{B^2+(1-C)^2}}}, $ $ \beta=\frac{1}{\sqrt{2}}\sqrt{1+\operatorname{sgn}(f)\frac{(1-C)}{\sqrt{B^2+(1-C)^2}}}$ we have that equation \eqref{p1} is reduced to the canonical form:
\begin{align}
\frac{\overline{x}^2}{\frac{1}{4}\left(\frac{\overline{D}^2}{\overline{A}^2} + \frac{\overline{E}^2}{\overline{A}\overline{C}}\right)} + \frac{\overline{y}^2}{\frac{1}{4}\left(\frac{\overline{D}^2}{\overline{A}\overline{C}} + \frac{\overline{E}^2}{\overline{C}^2}\right)}=1,
\end{align}
where
\begin{align}
\overline{A}=&\alpha^2 + B\alpha\beta + C\beta^2,\\
\overline{C}=&\beta^2-B\alpha\beta + C\alpha^2,\\
\overline{D}=&E\beta,\\
\overline{E}=&E\alpha.\label{eq:overE}
\end{align}
Note that $ \overline{A}=\left(\alpha + \frac{\beta}{2}B\right)^2  + \frac{(y_0-g_0)^2+2\vert y_0-g_0\vert}{\sigma_1^2}>0.$ Analogously, we can show that $ \overline{C}>0$.
From \eqref{eq:xf0neq0}-\eqref{eq:overE}, we have that the Funk's parabola of type 1 is a section of an Euclidean ellipse with center in $ \mathcal{C}=-\left(\alpha\frac{D}{2A} + \beta\frac{E}{2C}, \operatorname{sgn}(y_0-g_0)\left(1+ \beta\frac{D}{2A}+\alpha\frac{E}{2C}\right)\right) $ and vertices
$ \mathcal{C}\pm \frac{1}{2}\sqrt{\frac{\overline{D}^2}{\overline{A}^2}+\frac{\overline{E}^2}{\overline{A}\overline{C}}}\left(\alpha,\beta\operatorname{sgn}(y_0-g_0)\right),  $ 
$ \mathcal{C}\pm \frac{1}{2}\sqrt{\frac{\overline{D}^2}{\overline{A}\overline{C}}+\frac{\overline{E}^2}{\overline{C}^2}}\left(-\beta,\alpha\operatorname{sgn}(y_0-g_0)\right). $
\end{proof}  
\subsubsection{Funk's Parabola type 2 - Particular Case}
By Theorem \ref{maintheorem} and \eqref{dist_pt_reta_const} we have that the point $P=(x,\,y)$ satisfies (P2) if, and only if, \begin{align}\label{eq:tipe2norm}
\displaystyle\left\Vert \frac{\mathcal{F}}{\rho^{-1}(y,y_0)}-P\right\Vert^2 = \left\Vert \frac{Q}{\rho^{-1}(y,y_0)}-P \right\Vert^2,
\end{align}
where $ Q=\left(x.\rho^{-1}(y,y_0),y_0\right)$ and $\rho $ is given by \eqref{rxy}.\par
The development of equations \eqref{eq:tipe1norm} and \eqref{eq:tipe2norm} are analogous. Hence we have, first, the degenerate case of the Funk's parabola of type 2 $(y_0=g_0)$, is characterized by the line
\[x-f_0\rho(y,y_0)=0,\] which is just a reflection of a Funk's parabola of type 1 (degenerate) in relation to the line $s$.\par
Second, the non-degenerate Funk's parabolas of type 2 are characterized by the following equation.
\begin{propo}\label{prop1}
Let $s: y=y_0 \subset \mathbb{B}^2$ be a directrix, $\mathcal{F}=(f_0,g_0) \in \mathbb{B}^2$ the focus with $g_0\neq y_0$ and $ \mathcal{H} \subset \mathbb{R}^2$ an Euclidean hyperbola characterized by the equation.
\begin{align}\label{p2}
x^2 +B{x}{\overline{y}} +C{\overline{y}^2} +E{\overline{y}} = 0,
\end{align}
where 
\begin{align}
\overline{y}=&1-\operatorname{sgn}(y_0-g_0)y\\
B=&-\frac{2f_0}{\sigma_2},\label{defB2}\\
C=&\frac{f_0^2 + (y_0-g_0)^2-2\vert y_0-g_0\vert}{\sigma_2^2},\label{defC2}\\
E=&\frac{2\vert y_0-g_0\vert}{\sigma_2},\label{defE2}
\end{align}
and 
\begin{align}\label{defsigma2}
\sigma_2= 1-\operatorname{sgn}(y_0-g_0)y_0.
\end{align}
Then the locus of the Funk's parabola of type 2 is the convex curve \[ \mathcal{H} \cap \mathbb{B}^2. \]
\end{propo}
\begin{proof}
Note that the discriminant $I=B^2-4C$ is given by
\[I=4\vert y_0-g_0\vert(2-\vert y_0-g_0\vert) >0.\]
what shows us that the second-degree equation \eqref{p1} is hyperbolic in the Euclidean geometry (i.e., two concurrent lines, or a hyperbola in $ \mathbb{R}^2 $). As $ E\neq 0$, the two lines option is disregarded.\par
Analogous to what was done in Proposition \ref{prop1}, it is possible to show that the focus $\mathcal{F}$ and the points $P$ of the Funk's parabola of type 2 are in an open subspace generated by the directrix $s:y=y_0$. Therefore, only one of the curves of the Euclidean hyperbola intercepts the domain $ \mathbb{B}^2$. It remains to show that the set $ \mathcal{H} \cap \mathbb{B}^2$ its not empty, in fact, let $ Q*\in s $ be such that $d(F,s)=d(F,Q*)$, and $f:[0,1]\to \mathbb{R}$ the function defined by
\[f(t)=d(P_t,\mathcal{F})-d(P_t,Q*),\]
where $ P_t=\mathcal{F}+ t(Q*-\mathcal{F})\in \mathbb{B}^2$. The function $f$ is continuous and $f(0)=-d(P,Q*)<0, $ $ f(1)=d(Q*,\mathcal{F})>0$. By the Intermediate Value Theorem , we have that there exist $t_0\in (0,1)$ such that $f(t_0)=0.$
\end{proof}
\subsubsection{Examples - Funk's Parabolas type 1 and 2}
\begin{exa}
\textbf{Funk's parabolas of type 1 and 2 passing through the origin}
Let be the parabolas of type 1 and 2 passing through $(0,0)$, whose equations are given by \eqref{p1} and \eqref{p2}.\par
Substituting  $x=y=0$ in \eqref{p1} or \eqref{p2}, we obtain $C = -E$, in both cases, and we have that $C = -E$ if, and only if $\Vert \mathcal{F}\Vert^2 = y_0^2$. This equations says that, the Funk's parabolas of type 1 and 2 pass through the origin if, and only if, the Euclidean distance of $\mathcal{F}$ and $s$ to the origin are the same.
\begin{figure}[H]
\centering
\caption{Funk's parabola of type 1 with $\mathcal{F} = (0.3, -0.4)$ and $y_0 = 0.5$.}
\includegraphics[width=0.5\textwidth]{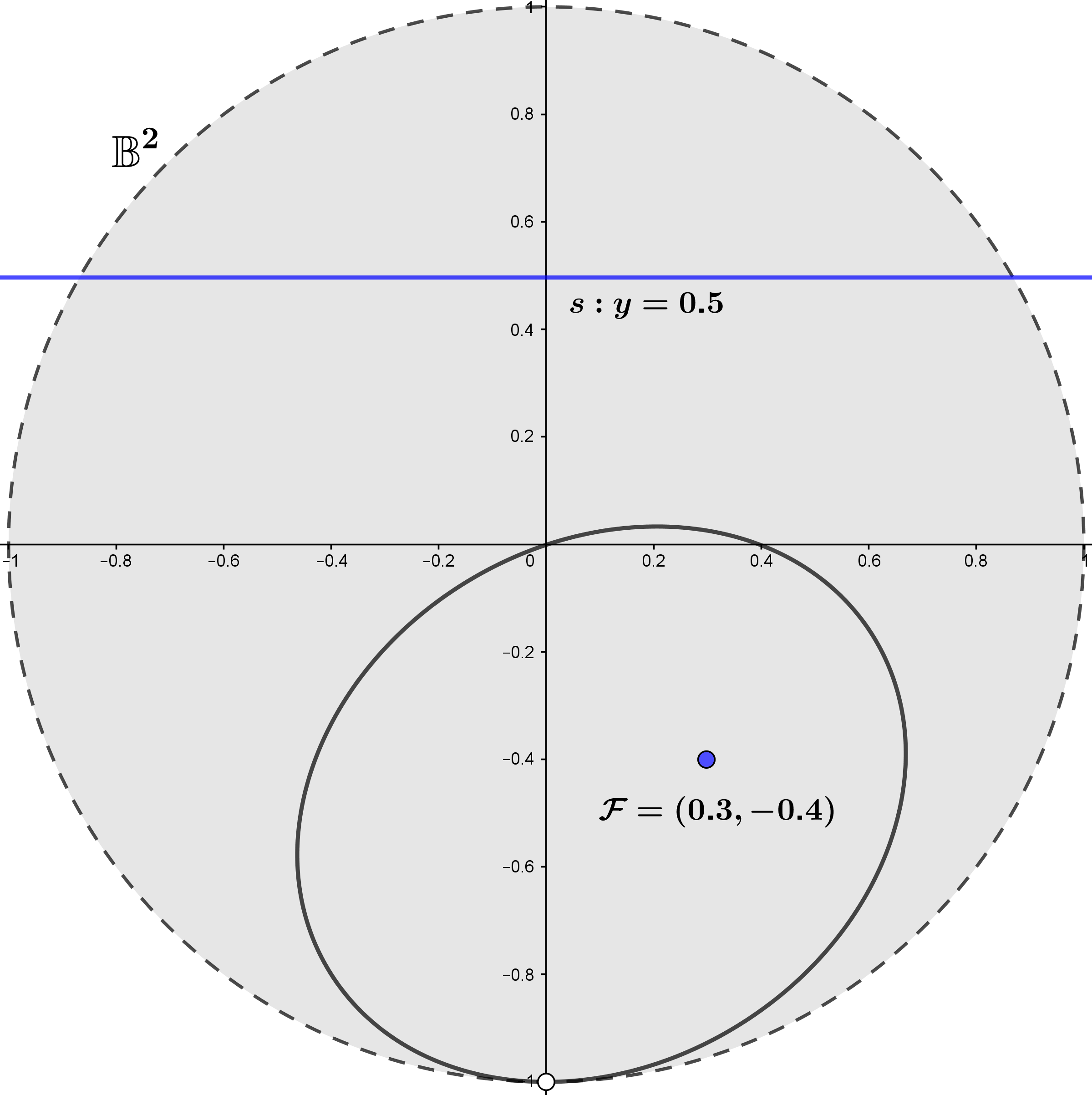}
\end{figure}                        
\begin{figure}[H]
\centering
\caption{Funk's parabola of type 1 with $\mathcal{F} = (0.3, 0.4)$ and $y_0 = - 0.5$.}
\includegraphics[width=0.5\textwidth]{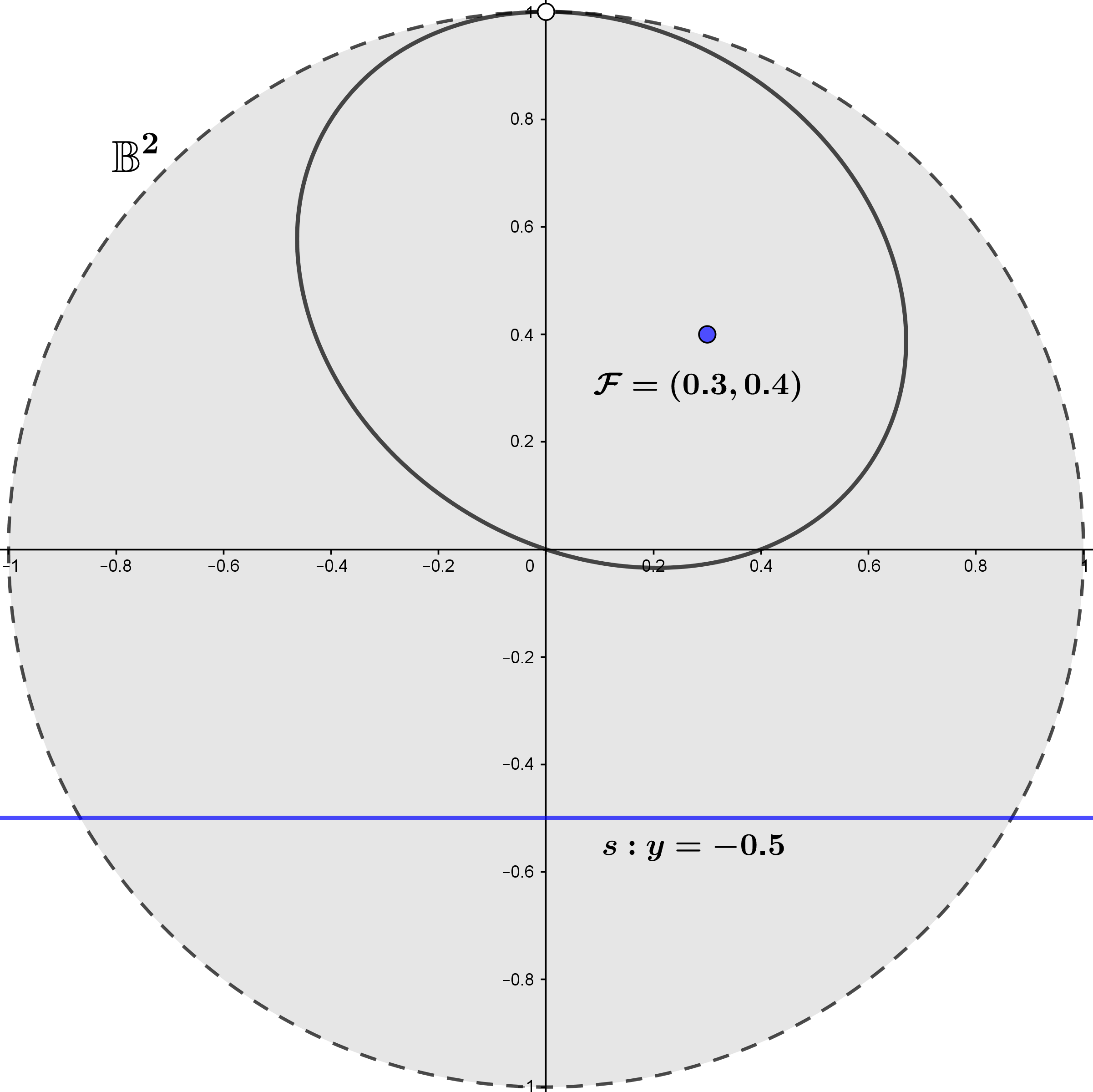}
\end{figure}
\begin{figure}[H]
\centering
\caption{Funk's parabola of type 2 with $\mathcal{F} = (0.3, - 0.4)$ and $y_0 = 0.5$.}
\includegraphics[width=0.5\textwidth]{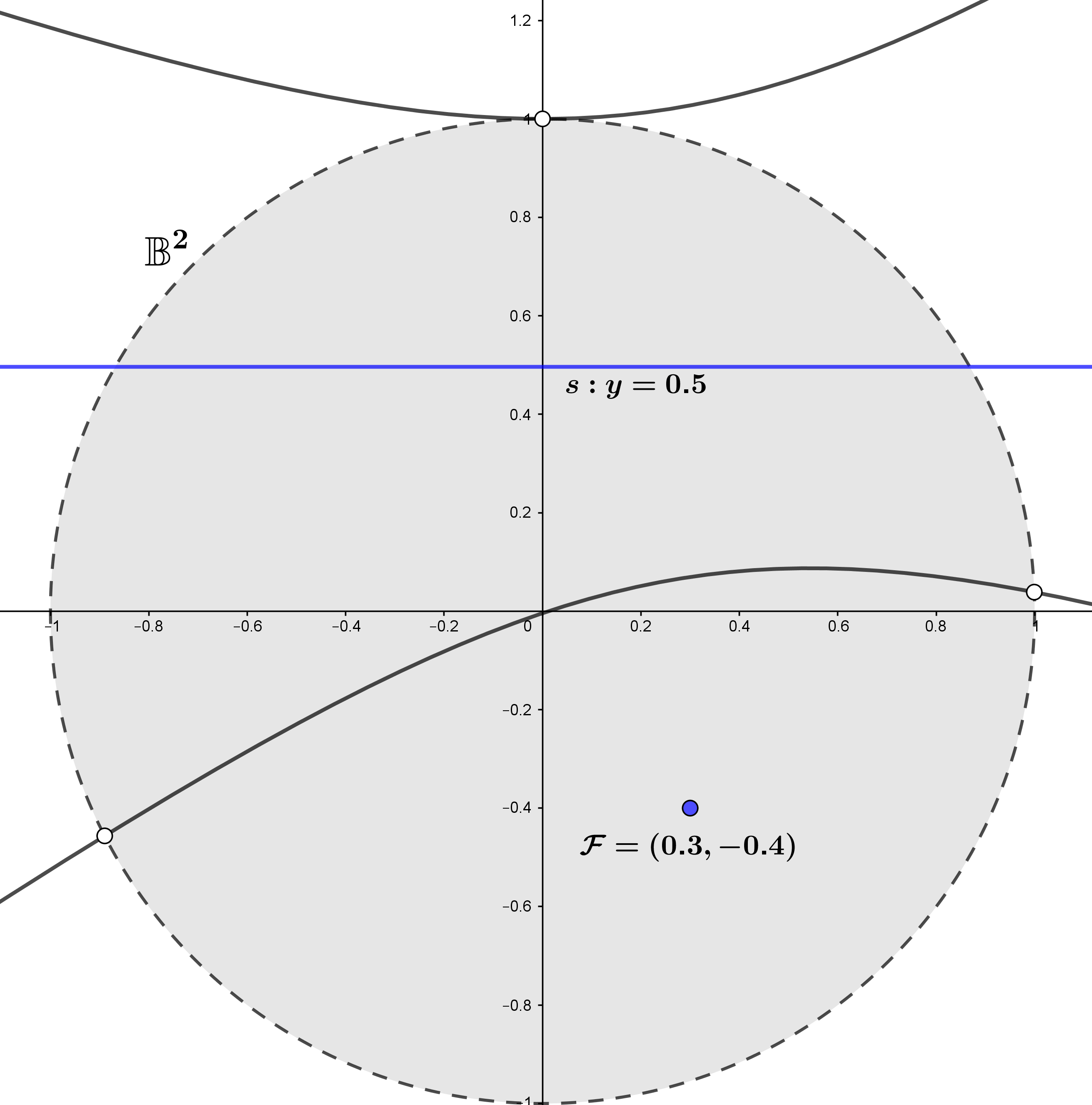}
\end{figure}                        
\begin{figure}[H]
\centering
\caption{Funk's parabola of type 2 with $\mathcal{F} = (0.3, 0.4)$ and $y_0 = - 0.5$.}
\includegraphics[width=0.5\textwidth]{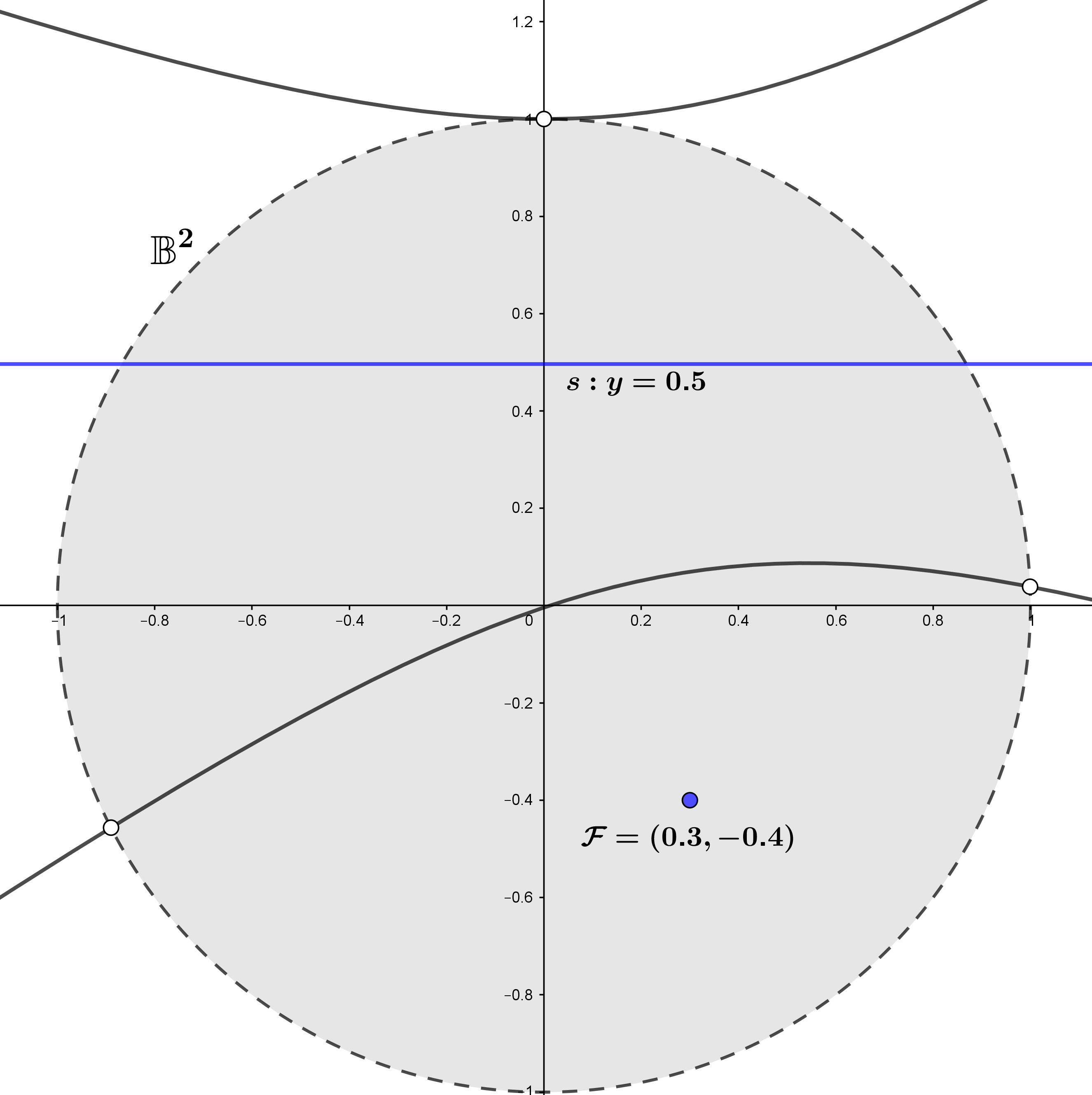}
\end{figure}                     
\end{exa}
\begin{exa}        
\textbf{Funk's parabolas of type 1 and 2 with focus in the origin.}

\end{exa}               



%
%
%

\subsubsection{Funk's Parabola type 3 - Particular Case}


By Theorem \ref{maintheorem} and \eqref{dist_pt_reta_const} we have that the point $P$ satisfies (P3) if, and only if,
\begin{align}\label{3p0}
\left\Vert\frac{\mathcal{F}}{\rho(y,y_0)}-P\right\Vert^{2}=\left\Vert\frac{P}{\rho(y,y_0)}-Q\right\Vert^{2},
\end{align}
where $ Q=(\frac{x}{\rho(y,y_0)},y_0)\in s $ and $ \rho $ is given by \eqref{rxy}.
Using the properties of absolute value and inner product, we obtain the following equation            
\eqref{3p0}, in the coordinates $\mathcal{F}=(f_0,g_0),\, P=(x,\,y)$ and $ s: y=y_0, $ is given by:
\begin{align}\label{3p1}
\overline{y}^4 + x^2\overline{y}^2-2\overline{y}^3+A\overline{y}^2 + Bx\overline{y} + C\overline{y}+D=0,
\end{align}
where  \begin{align*}
\overline{y}=&1-\operatorname{sgn}(y_0-y)\cdot y,\\
A=&2\sigma_3g_0\operatorname{sgn}(y_0-y),&
B=&-2\sigma_3 f_0,\\
C=&2\sigma_3(1-\operatorname{sgn}(y_0-y)\cdot g_0)&
D=&\sigma_3^2(f_0^2+g_0^2-1),
\end{align*}
and $ \sigma_3=1-\operatorname{sgn}(y_0-y)\cdot y_0. $

\subsubsection{Funk's Parabola type 4 - Particular Case} 
By Theorem \ref{maintheorem} and \eqref{dist_pt_reta_const} we have that the point $P$ satisfies (P4) if, and only if,
\begin{align}\label{4p0}
\left\Vert\frac{\mathcal{F}}{\rho^{-1}(y_0,y)}-P\right\Vert^{2}=\left\Vert\frac{P}{\rho^{-1}(y_0,y)}-Q\right\Vert^{2},
\end{align}
where $ Q=(\frac{x}{\rho^{-1}(y_0,y)},y_0)\in s $ and $ \rho $ is given by \eqref{rxy}.
Using the properties of absolute value and inner product, we obtain the equation \eqref{4p0}, in the coordinates $ \mathcal{F}=(f_0,g_0),\, P=(x,\,y) $ and $ s: y=y_0, $ is given by:
\begin{align}\label{4p1}
\overline{y}^4 + x^2\overline{y}^2-2\overline{y}^3+A\overline{y}^2 + Bx\overline{y} + C\overline{y}+D=0,
\end{align}
where \begin{align*}
\overline{y}=&1+\operatorname{sgn}(y_0-y)\cdot y,\\
A=&-2\sigma_4g_0\operatorname{sgn}(y_0-y),&
B=&-2\sigma_4 f_0,\\
C=&2\sigma_4(1+\operatorname{sgn}(y_0-y)\cdot g_0),&
D=&\sigma_4^2(f_0^2+g_0^2-1),
\end{align*}
and $ \sigma_4=1+\operatorname{sgn}(y_0-y)\cdot y_0. $
\begin{remark}\label{obs34equiv}
Note that the equation for parabolas of type 3 when $y < y_0$ (resp. $y>y_0$) is equal to the equation for parabolas of type 4 when $ y > y_0$ (resp. $ y<y_0 $).  	
\end{remark}
%

\subsubsection{Further Analysis of Funk's Parabolas type 3 and 4}

Due to Remark \ref{obs34equiv}, the graphs of parabolas of type 3 and 4 are equal, then we will study only the equation  \eqref{3p1}. 
As can be noticed, the equation \eqref{3p1} is a quartic equation. Now, we will show that the right side of \eqref{3p1} is irreducible. 
In order to do this, lets suppose that the equation \eqref{3p1} can be written as
\begin{align*}
\left(\overline{y}^2+Gx\overline{y} + Hx^2 + Ix+ J\overline{y} + K\right)\times\left(\overline{y}^2 + Lx\overline{y} + Mx^2+Nx + P\overline{y} + Q\right) = 0.
\end{align*}
Multiplying and grouping terms with $x$ and $y$ above we have            
\begin{align*}
0=&\overline{y}^4+(L+G)\overline{y}^3x + (M+H+GL)\overline{y}^2x^2+(MG+HL)\overline{y}x^3 + HMx^4+\\
&+(N+I+GP+JL)\overline{y}^2x + (P+J)\overline{y}^3 + (NG + JM+HP+IL)\overline{y}x^2 + \\
&+(HN+IM)x^3 + (Q+K+JP)\overline{y}^2 + (GQ+IP+LK+NJ)x\overline{y} + \\
&+(QH+NI+MK)x^2 + (JQ+PK)\overline{y} + (KN+IQ)x + KQ.      
\end{align*}
By comparing the equations above with \eqref{3p1} we obtain 14 equations, of which we are only interested in
\begin{align}
HM=0,\quad L+G=0 \quad M+H=1,\quad MG+HL=0,\quad HN+IM=0\label{eq001}\\
NG+JM+HP+IL=0,\quad QH+NI+MK=0 \quad JQ+PK=C.\label{eq002}
\end{align}
From  \eqref{eq001} and \eqref{eq002}, we have  $C=2\sigma_3(1-\operatorname{sgn}(y_0-y)\cdot g_0)=0, $ which is a contradiction.\par
On the other side, suppose that the equation \eqref{3p1} can be written as  
\begin{align*}
\left(\overline{y}^3+G\overline{y}^2x+H\overline{y}x^2 + Ix^3+J\overline{y}^2+K\overline{y}x + Lx^2+My+Nx+P\right)(\overline{y} + Qx+R) = 0,
\end{align*}            
%
Doing the multiplications on the left side, and comparing with \eqref{3p1}, we obtain 14 equations:                 
\begin{align}
IQ=0, \quad Q+G=0, \quad HQ+I=0, \quad GR+JQ+K=0 , \quad QG+H=1\label{eq003}\\R+J=-2,
\quad HR+KQ+L=0\quad LR+NQ=0, \label{eq004}\\IR+LQ=0, \quad JR+M=A,\quad                 	KR+MQ+N=B,\quad  \label{eq005}\\
\quad MR+P=C, \quad NR+PQ=0, \quad PR=D
\end{align}
Three cases can be considered: when $I = Q = 0,$ or $I=0,\;Q\neq0$ or $I\neq0,\;Q=0$.\medskip\\
\noindent\textbf{Case 1: $I=Q=0$.}\par
From \eqref{eq003} we have $ G=0, K=0, H=1, $ then, from \eqref{eq004} we have $ R+L=RL=0, $ from the last equation in \eqref{eq005} we conclude that $ D=0, $ which is absurd.\smallskip\\
\noindent \textbf{Case 2: $I=0$ and $Q\neq0$.}\par
From \eqref{eq003} we have $ H=0 $ and $ Q+G=0, QG=1, $ then there no exist $ Q $ and $ G $ real numbers, satisfying these conditions.\smallskip\\
%
%
%
\textbf{Caso 3: $I\neq0$ e $Q=0$.}\par
From the third equation in \eqref{eq003} we have $I=0,$ which is absurd.\medskip\\
Therefore, we conclude that it is impossible to write the right side of equation \eqref{3p1} as a product of lower degree polynomials.
\subsubsection{Examples - Funk's Parabolas type 3 and 4}
\begin{exa}\label{par3ex1}\hfill 
\begin{figure}[H]
\centering
\caption{Funk's Parabolas type 3 with $\mathcal{F} = (0, 0)$ and $y_0 = 0.3$.}
\includegraphics[width=0.5\textwidth]{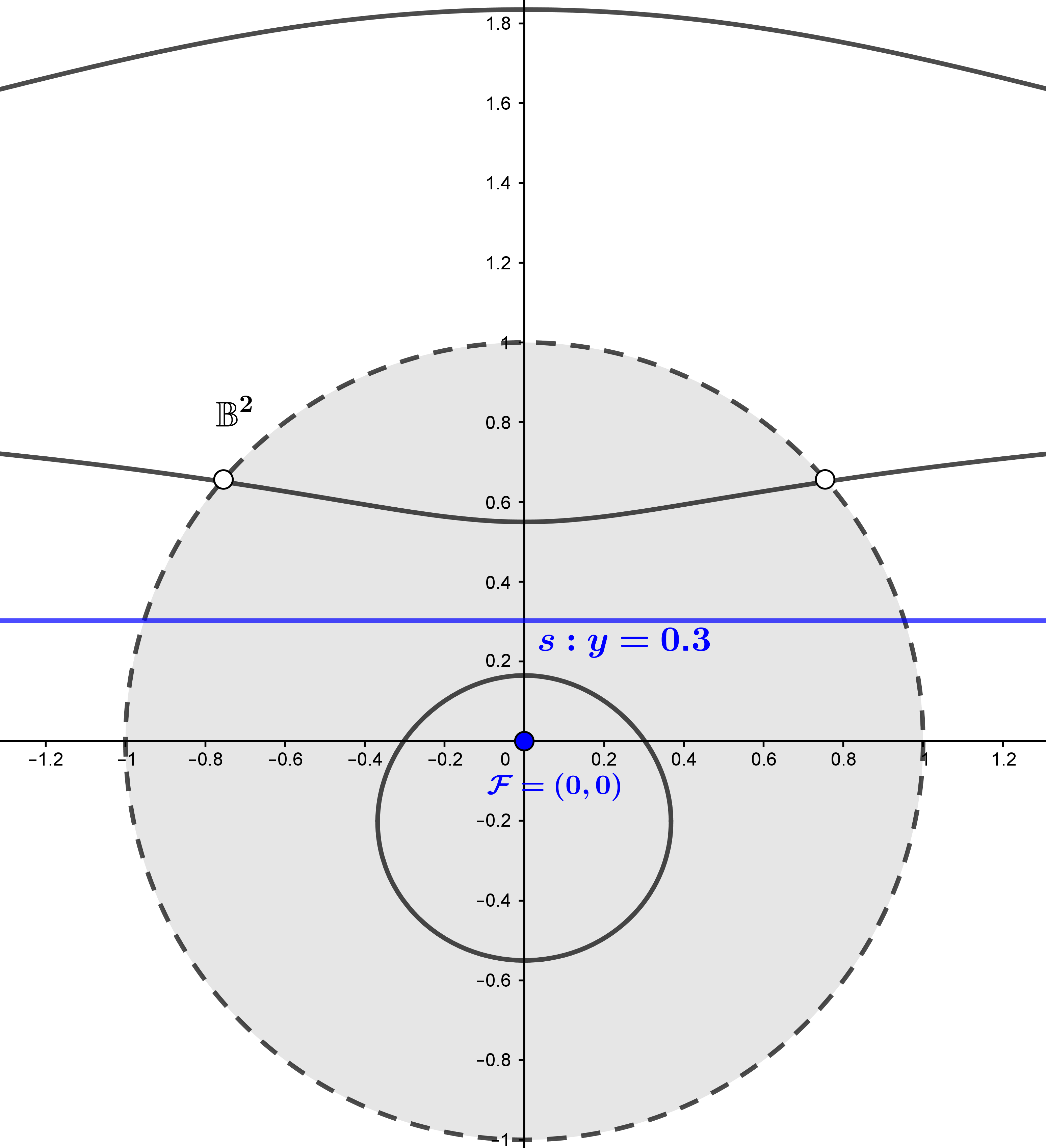}
\end{figure}                
\begin{figure}[H]
\centering
\caption{Funk's Parabolas type 3 with $\mathcal{F} = (0.6, 0)$ and $y_0 =  0.3$.}
\includegraphics[width=0.5\textwidth]{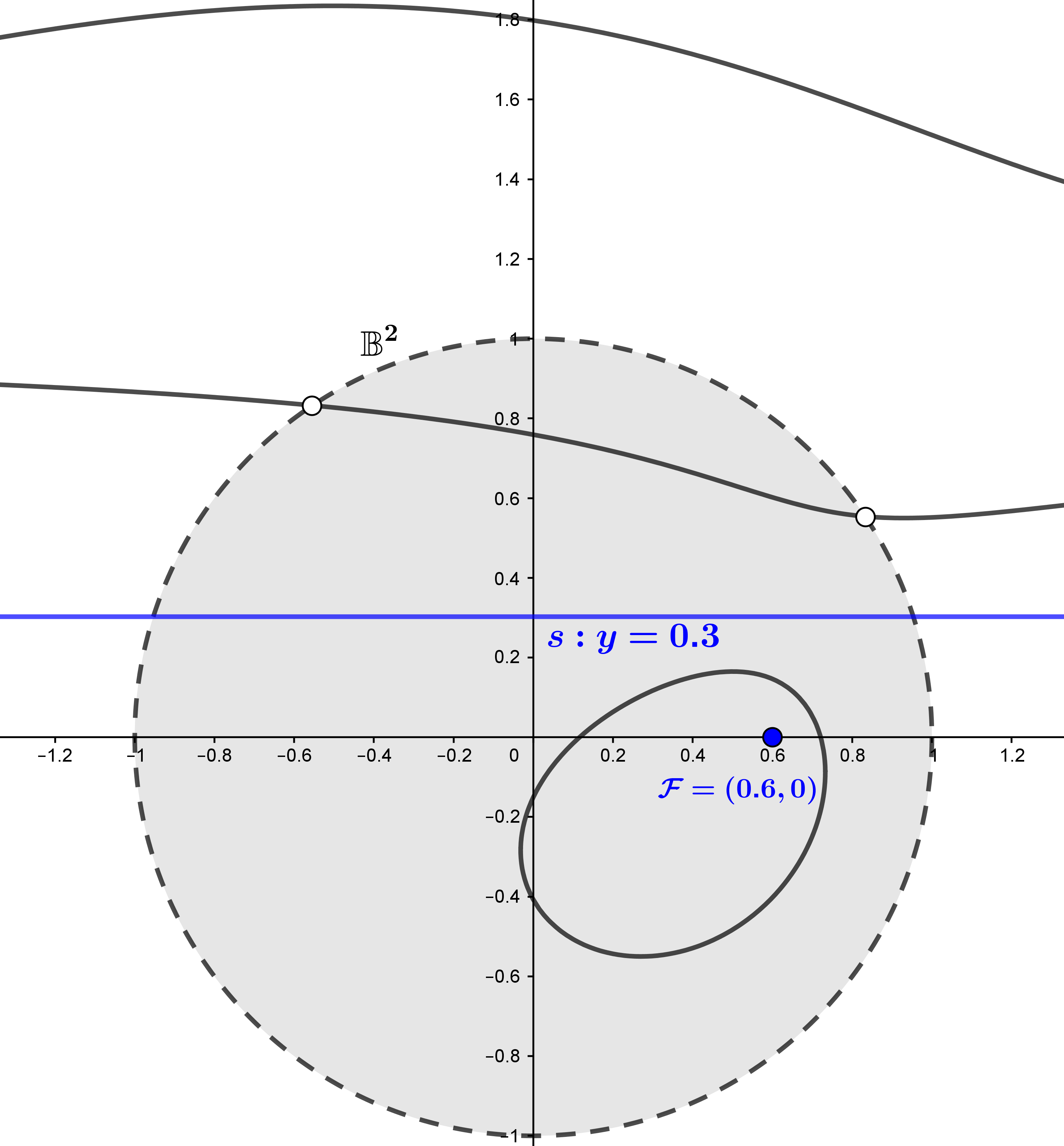}
\end{figure}
\begin{figure}[H]
\centering
\caption{Funk's Parabolas type 3 with $\mathcal{F} = (0.7, 0.3)$ and $y_0 =  0.3$.}
\includegraphics[width=0.5\textwidth]{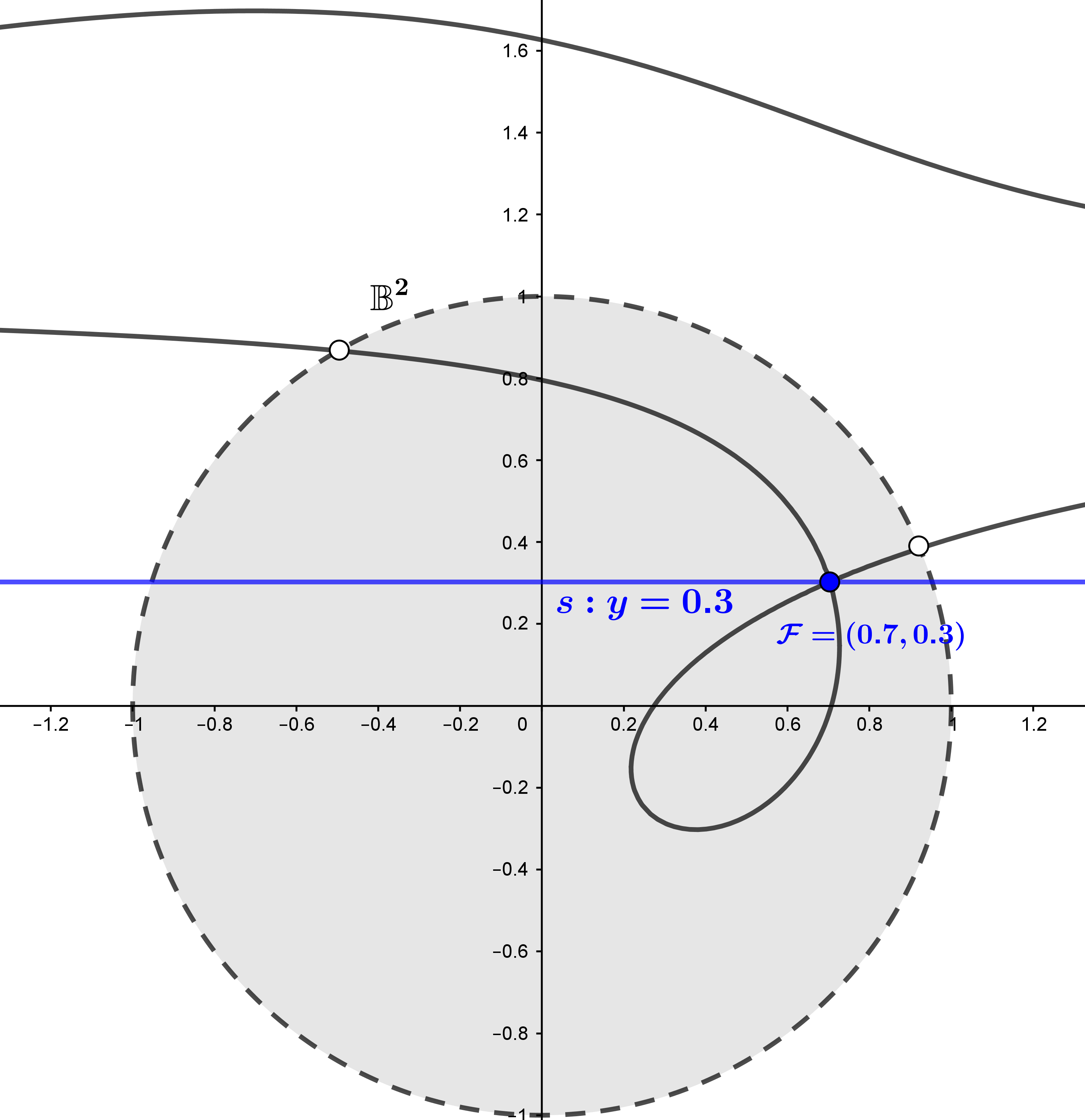}
\end{figure}                    
\begin{figure}[H]
\centering
\caption{Funk's Parabolas type 3 with $\mathcal{F} = (0, -0.4)$ and $y_0 =  -0.8$.}
\includegraphics[width=0.5\textwidth]{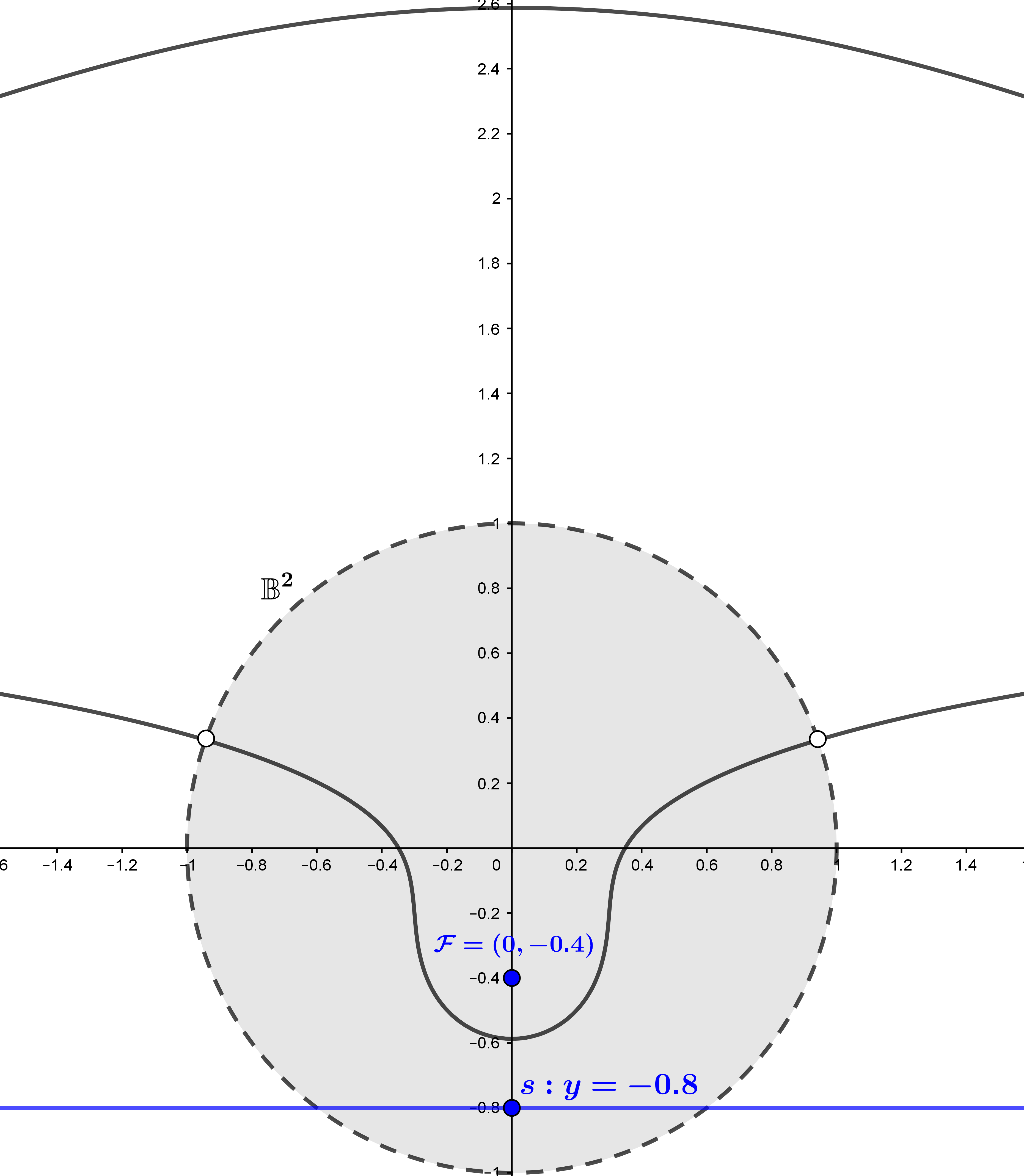}
\end{figure}                    
\end{exa}
%
%
%
%
%

\subsection*{Acknowledgements} We want to thank to Jonny Ardila for its help in the writing of this work.


\end{document}